\documentclass[a4paper,10pt]{article}
\pdfoutput=1

\usepackage[utf8]{inputenc}

\usepackage{amsmath,amsfonts,amsthm,amssymb}
\usepackage{graphicx}
\usepackage[colorlinks,citecolor=blue,linkcolor=blue,urlcolor=blue]{hyperref}
\usepackage[nameinlink,noabbrev,capitalize]{cleveref}
\usepackage{tabularx}
\usepackage{doi}
\numberwithin{equation}{section}

\newtheorem{theorem}{Theorem}[section]
\newtheorem{lemma}[theorem]{Lemma}

\newtheorem{remark}[theorem]{Remark}
\newtheorem{corollary}[theorem]{Corollary}

\newtheorem{algorithm}{Algorithm}
\newtheorem{assumption}{Assumption}

\newcommand\be{\begin{equation}}
\newcommand\ee{\end{equation}}

\renewcommand{\subset}{\subseteq}

\newcommand{\dx}{\,\text{\rm{}d}x}
\newcommand{\dy}{\,\text{\rm{}d}y}

\def\N{\mathbb  N}
\def\R{\mathbb  R}

\renewcommand{\phi}{\varphi}

\newcolumntype{L}{>{$}l<{$\quad}}
\newcolumntype{R}{>{$}r<{$\quad}}
\newcolumntype{C}{>{$}c<{$}}

\renewcommand{\theenumi}{(\arabic{enumi})}

\title{
Sparse optimization problems in fractional order Sobolev spaces
\footnote{HA was partially supported by the US National Science Foundation (NSF) grants
DMS-2110263, DMS-1913004, the Air Force Office of Scientific Research (AFOSR)
under Award NO: FA9550-19-1-0036. DW was partially supported by the German
Research Foundation DFG under project grant Wa 3626/3-2.}
}

\author{
Harbir Antil
\footnote{Department of Mathematical Sciences and the Center for Mathematics and
Artificial Intelligence (CMAI), George Mason University, Fairfax, VA 22030, USA, {\tt hantil@gmu.edu}.
},
Daniel Wachsmuth%
\footnote{Institut f\"ur Mathematik,
Universit\"at W\"urzburg,
97074 W\"urzburg, Germany, {\tt daniel.wachsmuth@mathematik.uni-wuerzburg.de}.
}}

\begin{document}

\maketitle

\paragraph{Abstract.}
We consider
optimization problems in the fractional order Sobolev spaces %$H^s(\Omega)$, $s\in (0,1)$,
with sparsity promoting objective functionals containing $L^p$-pseudonorms, $p\in (0,1)$.
Existence of solutions is proven. By means of a smoothing scheme, we obtain first-order optimality conditions,
which contain an equation with the fractional Laplace operator.
An algorithm based on this smoothing scheme is developed. Weak limit points of iterates are shown
to satisfy a stationarity system that is slightly weaker than that given by  the necessary condition.



{\color{red}
\paragraph{Note.}
This is a correction of the published version\footnote{see Inverse problems 39, 044001 (2023), \doi{10.1088/1361-6420/acbe5e}}.
In the original publication, the result of \cref{lem_boundedl1_lambda} was wrong as stated.
The result can be corrected by using a slightly different penalization approach in \cref{sec_optimality}.
Differences to the published version are marked in red.
The authors wish to thank Anna Lentz for pointing out this error.
}

\section{Introduction}

We are interested in the following optimization problem
\begin{equation}\label{eq_problem}
\min_{u\in H^s(\Omega)} F(u) + \frac\alpha2 \|u\|_{H^s(\Omega)}^2 + \beta  \|u\|_p^p,
\end{equation}
where
$F: H^s(\Omega) \to \R$ is assumed to be smooth,
and
\begin{equation}\label{eq_lpnorm}
	\|u\|_p^p:= \int_\Omega |u|^p \dx
\end{equation}
with $p\in (0,1)$
is the $L^p$-pseudo-norm of $u$, which is a concave and nonsmooth function.
Here, $\Omega\subset\R^d$ is a Lipschitz domain, $\alpha>0$, $\beta>0$.
Moreover, we will work with $s\in (0,1)$, where $s$ is chosen small to
facilitate discontinuous solutions.

The main motivation of this work comes from sparse control problems: we want to find controls with small support.
This can be useful in actuator placement problems, where one wants to identify small subsets of $\Omega$, where actuators needs to be placed.
The functionals of type \eqref{eq_lpnorm} are known to be sparsity promoting: solutions tend to be zero on certain parts of the domain.
Another motivation is the study of sparse source or coefficient identification problems, see, e.g. \cite{AsMannRosch2015,FanJiaoLuSun2016,JinMaass2012}.
In order to avoid over-smoothing, we study the regularization in $H^s(\Omega)$ for small $s\in (0,1)$.

If problem \eqref{eq_problem} would be set in $L^2(\Omega)$ instead of $H^s(\Omega)$, $s>0$, then it is impossible
to prove existence of solutions \cite{ItoKunisch2014}. In fact, following the construction in \cite{Wachsmuth2019},
one can construct problems of type \eqref{eq_problem} that do not attain their infimum on $L^2(\Omega)$.
The situation changes on $H^s(\Omega)$, $s>0$. Due to the compactness of the embedding $H^s(\Omega) \hookrightarrow L^2(\Omega)$, the
map $u\mapsto \|u\|_p^p$ is weakly continuous from $H^s(\Omega)$ to $\R$, which enables the standard existence proof, see \cref{thm_exist_sol}.

Since first introduced in \cite{HAntil_SBartels_2017a} for image denoising and phase-field
models, the use of $H^s(\Omega)$-norm as a regularizer, especially in the imaging community,
has seen a significant growth.
The use of functionals of type \eqref{eq_lpnorm} in imaging problems set on sequence spaces
has resulted in a rich literature. We refer only to \cite{ItoKunisch2014lp,Nikolova2005,NikolovaNgZhangChing2008}.
There it was proven that solutions are sparse: based on optimality conditions one can prove that entries $u_i$ of a solution $u\in \ell^2$ are either zero or have absolute value larger than some given number.

In the context of optimal control, the research on problems with functionals of type \eqref{eq_lpnorm} was initiated in \cite{ItoKunisch2014}.
There, for problems set on $L^2(\Omega)$, an optimality condition in form of the Pontryagin maximum principle was proven, which
can be used to prove sparsity of solutions. In addition, the regularization of these problems in $H^1(\Omega)$ was suggested.

In order to solve problems of type \eqref{eq_lpnorm} a monotone algorithm for a smoothing of \eqref{eq_lpnorm} was introduced in \cite{ItoKunisch2014,ItoKunisch2014lp}.
For $\epsilon>0$, one defines
\[
 \psi_{\epsilon}(t) = \begin{cases} \frac p2 \frac t{\epsilon^{2-p}} + (1-\frac p2)\epsilon^p & \text{ if } t\in [0,\epsilon^2),\\
                       t^{p/2} &\text{ if } t \ge \epsilon^2.
                      \end{cases}
\]
Then a smooth version of \eqref{eq_lpnorm} is given by $u\mapsto \int_\Omega \psi_\epsilon(u(x)^2)\dx$, see also \cite{SongBabuPalomar2015}.
It is proven in \cite{ItoKunisch2014,ItoKunisch2014lp} that for quadratic $F$ and fixed $\epsilon>0$ the weak limits of the iterates of the algorithm
solve the necessary condition of the smoothed problem.
In this paper, we extend this idea to allow for a decreasing sequence of smoothing parameters $\epsilon_k \searrow0$.
Still we can prove that weak limit points of iterates satisfy a certain optimality system.

In addition, we prove necessary conditions of first order for \eqref{eq_problem}. Here, we use a standard penalization technique together with the smoothing procedure
outlined above. Doing so, we are able to prove that for every local solution $\bar u \in H^s(\Omega)$
there is $\bar\lambda \in (H^s(\Omega))^*$ such that the system
\begin{align*}
\alpha (-\Delta)^s \bar u + \beta\bar\lambda & = -F'(\bar u)\\
\langle \bar\lambda, \bar u\rangle_{(H^s(\Omega))^*, H^s(\Omega)}  & =  p \int_\Omega |\bar u|^p \dx
\end{align*}
is satisfied.
For the precise statement of the result, we refer to \cref{thm_opt_cond}.
Let us mention that this result is new, even in the case  $s=1$.

Weak limit points of iterates of our algorithm, see \cref{alg} in \cref{sec_algorithm}, are proven to satisfy
for some $\bar\lambda \in (H^s(\Omega))^*$ the weaker system
\begin{align*}
\alpha (-\Delta)^s \bar u + \beta\bar\lambda & = -F'(\bar u)\\
\langle \bar\lambda, \bar u\rangle_{(H^s(\Omega))^*, H^s(\Omega)}  & \ge  p \int_\Omega |\bar u|^p \dx,
\end{align*}
see \cref{thm_optcon_lim_iter}.

The plan of the paper is as follows. We will work in an abstract framework with some Hilbert space $V$, where we have in mind to use $V=H^s(\Omega)$.
The proof of existence of solutions of \eqref{eq_problem} is given in \cref{sec_existence}.
The smoothing approach is described in \cref{sec_smoothing},
which is then used in \cref{sec_optimality} to obtain the first-order necessary optimality condition.
We comment on the use of the fractional Hilbert spaces $H^s(\Omega)$ and the possible realizations of fractional Laplace operators $(-\Delta)^s$ in \cref{sec_fractional}.
The optimization method is analyzed in \cref{sec_algorithm}.

\subsection*{Notation and assumptions}

We will consider the problem in a more general framework that includes $V=H^s(\Omega)$ and $V=H^1(\Omega)$.

\begin{assumption}[Standing assumption] \phantom{\quad}
	\label{ass_standing}

\begin{enumerate}
\item
\label{ass1_V}
$V$ is a real Hilbert space, $V\subset L^2(\Omega)$ with compact and dense embedding $V\hookrightarrow L^2(\Omega)$.
The inner product of $V$ is denoted by $\langle\cdot,\cdot\rangle_V$.
The duality product between $V^*$ and $V$ is denoted by  $\langle\cdot,\cdot\rangle_{V^*,V}$.

	\item $F:V \to \R$ is weakly lower semicontinuous, and bounded below by an affine function, i.e., there is $g\in V^*$, $c\in \R$ such that $F(u) \ge \langle g,u\rangle_{V^*,V} + c$
	for all $u\in V$.

	\item $F$ is continuously Fr\'echet differentiable %(do we need the stronger condition: from $L^2(\Omega)$ to $\R$?)
\end{enumerate}
\end{assumption}

We will employ the usual Hilbert (or Gelfand) triple setting $V \hookrightarrow L^2(\Omega) = L^2(\Omega)^* \hookrightarrow V^*$.

\section{Application to a sparse source identification problem}

Before investigating the abstract problem in depth, let us introduce one possible application of the sparse optimization problem.
Here, we will look into identification of sparse perturbations of the initial condition in a parabolic partial differential equation.
This is motivated by the question of detecting the source of, e.g., environmental pollution \cite{BieglerGhattasHeinkenschlossVanBloemenWaanders2003}.
The identification problem now reads:
Given some measurement $z$ of the state $y(T)$ at the terminal time $T>0$ determine the perturbation $u$ in the initial condition $y(0)=y_0 + u$.
We formulate this as the following minimization problem:
\[
 \min \frac 12 \|y(T)-z\|_{L^2(\Omega)} + \frac\alpha2 \|u\|_{H^s(\Omega)}^2 + \beta  \|u\|_p^p
\]
subject to
\[
 y_t - \operatorname{div}(a \nabla y) + f(y) = 0 \quad \text{ in } (0,T) \times \Omega,
\]
\[
 y |_{(0,T) \times \partial \Omega}=0,
\]
and
\[
 y(0)= y_0 + u.
\]
Here, $\Omega\subset \R^d$, $d=2,3$, is a bounded domain, $a\in L^\infty(\Omega)$ is a positive diffusivity coefficient, and
$f: \R\to \R$ is a smooth function such that $f'$ is bounded from below.
Here, $a$ and $f$ are assumed to be known.
Let us denote $F(u):= \frac 12 \|y(T)-z\|_{L^2(\Omega)}$. Then one can show that $F$ satisfies \cref{ass_standing} for $V=H^s(\Omega)$ with $s\ge0$, see, e.g., \cite{Troltzsch2010}
and the recent contribution \cite{CasasWachsmuth2022}.
Similar problems were investigated in \cite{CasasVexlerZuazua2015,LeykekhmanVexlerWalter2020}, where the unknown $u$ is a measure.
Following the considerations in \cite{Grasmair2009,GrasmairHaltmeierScherzer2008} one can consider noisy measurements $z^\delta$ with $\|z-z^\delta\|_{L^2(\Omega)} \le \delta$,
where $z$ denotes the unavailable exact data, and prove convergence of solutions of the problem above for $(\alpha,\beta,\delta)\to(0,0,0)$.

\section{Existence of solutions}
\label{sec_existence}

Let us first prove existence of solutions for the nonsmooth optimization problem

\begin{equation}\label{eq001}
\min_{u\in V} F(u) + \frac\alpha2 \|u\|_V^2 + \beta  \|u\|_p^p.
\end{equation}

\begin{theorem}\label{thm_exist_sol}
	The optimization problem \eqref{eq001} is solvable.
\end{theorem}
\begin{proof}
The proof follows by standard arguments. Let $(u_n)$ be a minimizing sequence. Due to \cref{ass_standing}, $(u_n)$ is bounded in $V$. Hence, we have (after extracting
a subsequence if necessary) $u_n \rightharpoonup\bar u$ in $V$ and  $u_n \rightharpoonup\bar u$ in $L^2(\Omega)$.
Passing to the limit (limit inferior) in the functional shows that $\bar u$ realizes the minimum.
\end{proof}

\begin{remark}\label{rem_origin_local_sol}
The origin $u_0=0$ is a minimum of \eqref{eq001} along lines. In fact, let $u\in V$. Then for $t>0$ small we have
\[
 \left( F(tu) + \frac\alpha2 \|tu\|_V^2 + \beta  \|tu\|_p^p \right) - F(0) =  t F'(0)u + o(t) + t^2 \frac\alpha2 \|u\|_V^2 + \beta t^p \|u\|_p^p,
\]
which implies
\[
 \lim_{t\searrow0}\frac1{t^p} \left(\left( F(tu) + \frac\alpha2 \|tu\|_V^2 + \beta  \|tu\|_p^p \right) - F(0) \right) =  \beta  \|u\|_p^p >0.
\]
Hence, the function $t\mapsto F(tu)$ has a local minimum at $t=0$.

A stronger claim holds true for minimization problems on $\mathbb R^n$ involving the $l^p$-norm
of vectors: there the origin is a local solution,
which is due to the inequality $\sum_{i=1}^n |x_i|^p \ge (\sum_{i=1}^n |x_i|)^p$ for $p\in (0,1)$ and $x\in \R^n$.
\end{remark}

\section{Smoothing scheme}
\label{sec_smoothing}

In order to prove necessary optimality conditions and to devise an optimization algorithm,
we will use a smoothing scheme, which was already employed in \cite{ItoKunisch2014,SongBabuPalomar2015}.
Let $\epsilon\ge0$.
Then we will work with the following smooth approximation of $t\mapsto t^{p/2}$ defined by
\[
 \psi_{\epsilon}(t) = \begin{cases} \frac p2 \frac t{\epsilon^{2-p}} + (1-\frac p2)\epsilon^p & \text{ if } t\in [0,\epsilon^2),\\
                       t^{p/2} &\text{ if } t \ge \epsilon^2
                      \end{cases}
\]
with derivative given by
\[
 \psi'_{\epsilon}(t) = \frac p2 \min(\epsilon^{p-2},t^{\frac{p-2}2}).
\]
Note that $\psi_0(u^2) = |u|^p$.
In addition, we have the following properties of $\psi_\epsilon$.
\begin{lemma}\label{lem_psi_properties}
Let $p\in(0,2)$, $\epsilon>0$ be given. Then $\psi_\epsilon$ satisfies the following:
\begin{enumerate}
	\item \label{item_prop_psi_1}  $\psi_{\epsilon} : [0,\infty) \to [0,\infty)$ is continuously differentiable with $0\le \psi_\epsilon' \le \frac p2\epsilon^{p-2}$,
	\item \label{item_prop_psi_2}  $\psi_{\epsilon}$ is concave,
	\item \label{item_prop_psi_3}  $|t|^p \le \psi_\epsilon(t^2) \le |t|^p + (1-\frac p2)\epsilon^p$ for all $t\in \R$.
	\item \label{item_prop_psi_4}  $\epsilon \mapsto \psi_{\epsilon}(t) $ is monotonically increasing for all $t\ge0$.
\end{enumerate}
\end{lemma}
\begin{proof}
Claims \ref{item_prop_psi_1}, \ref{item_prop_psi_2}, and the second inequality of \ref{item_prop_psi_3}  can be verified by elementary calculations.
Let us prove the first inequality of \ref{item_prop_psi_3}. Due to the concavity of $t\mapsto t^{p/2}$, we obtain for $|t| \le \epsilon$
\[
(t^2)^{\frac p2} \le (\epsilon^2)^{\frac p2}  + \frac p2 (\epsilon^2)^{\frac p2-1} ( t^2 - \epsilon^2) = \psi_\epsilon(t^2),
\]
which is the claim. In order to prove claim \ref{item_prop_psi_4} let $0<\epsilon_1<\epsilon_2$ and $t\ge0$. Then $\psi_{\epsilon_1}(t) \le \psi_{\epsilon_2}(t)$
is trivial for $t\ge\epsilon_2^2$.
On the interval $[\epsilon_1^2,\epsilon_2^2]$,
the function $\psi_{\epsilon_2}$ is affine linear and tangent to $\psi_{\epsilon_1}$, hence the claim follows by concavity of $t\mapsto t^{p/2}$.
On the interval $(0,\epsilon_1^2)$ both functions are affine linear with $\psi_{\epsilon_2}' < \psi_{\epsilon_1}'$.
In addition, we have $\psi_{\epsilon_2}(\epsilon_1^2) \ge \psi_{\epsilon_1}(\epsilon_1^2)$,
which implies $\psi_{\epsilon_1}(t) \le \psi_{\epsilon_2}(t)$ for $t\in (0,\epsilon_1^2)$.
\end{proof}

Let us define the following integral function associated to $\psi_{\epsilon}$
\[
 G_\epsilon(u):= \int_\Omega \psi_{\epsilon}(|u|^2)\dx,
\]
which serves as an approximation of $\int_\Omega |u|^p \dx$.

\begin{lemma}\label{lem_conv_psi_eps}
	Let $(u_k)$ and $(\epsilon_k)$ be sequences such that $u_k \to u$ in $L^1(\Omega)$  and $\epsilon_k\to \epsilon\ge0$.
	Then $G_{\epsilon_k}(u_k) \to G_\epsilon(u)$.
\end{lemma}
\begin{proof}
	This is \cite[Lemma 5.1 and (5.3)]{ItoKunisch2014}.
\end{proof}

\begin{lemma}\label{lem_diff_Geps}
	Let $\epsilon>0$. Then $u\mapsto G_\epsilon(u)$ is Fr\'echet differentiable from $L^2(\Omega)$ to $\R$, with derivative given by
	\[
		G_\epsilon'(u)h = \int_\Omega 2u(x)\cdot \psi_\epsilon'(u(x)^2) \cdot h(x)\ \dx.
	\]
\end{lemma}
\begin{proof}
	Fix $\epsilon>0$.
	Let us define the Nemyzki (or superposition) operator $\Psi_\epsilon(u)$  by $\Psi_\epsilon(u)(x) := \psi_\epsilon(u(x)^2)$.
	Due to the growth estimate of \cref{lem_psi_properties}, $\Psi_\epsilon$ is a mapping from $L^2(\Omega)$ to $L^{1/p}(\Omega)$,
	hence continuous by \cite[Thm.\@ 3.1]{appellzabrejko}. Note that $1/p>1$.
	Due to the boundedness of $\psi_\epsilon'$, the Nemyzki operator induced by $u\mapsto 2u\cdot \psi_\epsilon'(u^2)$ is continuous from $L^2(\Omega)$ to $L^2(\Omega)$.
	This implies the Fr\'echet differentiability of $\Psi_\epsilon$ from $L^2(\Omega)$ to $L^1(\Omega)$, see \cite[Thm.\@ 3.12]{appellzabrejko},
	and the claim follows.
\end{proof}

\section{Optimality conditions}
\label{sec_optimality}

Here, we will now prove optimality conditions for the nonsmooth problem \eqref{eq001}.
To this end, we will introduce an auxiliary smooth optimization problem.
For a general exposition of this method, we refer to \cite[Section 3.2.2]{Barbu1993}.
Given $\epsilon \ge0$, define
\begin{equation}\label{eq_def_Phi_eps}
	\Phi_\epsilon(u) := F(u) + \frac\alpha2 \|u\|_{V}^2 + \beta G_\epsilon(u).
\end{equation}
Let now $\bar u \in V$ be a local solution of the original problem \eqref{eq001}. Then there is $\rho>0$ such that
\[
\Phi_0(\bar u) \le \Phi_0(u)
\]
for all $u$ with $\|u-\bar u \|_{V} \le \rho$.
For $\epsilon>0$ define the auxiliary problem:
\begin{equation}\label{eq_aux_Phi_eps}
 \min \Phi_\epsilon(u) + \frac12 \|u-\bar u\|_{{\color{red} L^2(\Omega)}}^2 \quad \text{ subject to } \|u-\bar u \|_{V} \le \rho.
\end{equation}
Arguing as in the proof of \cref{thm_exist_sol}, one can prove that
\eqref{eq_aux_Phi_eps} is solvable.
We will now show that solutions to this problem converge for $\epsilon\searrow0$ to solutions of \eqref{eq001}.
\begin{lemma}\label{lem_conv_aux_solutions}
Let $(\epsilon_k)$ be a sequence of positive numbers with $\epsilon_k\to 0$.
Let $(u_k)$ be a family of global solutions of \eqref{eq_aux_Phi_eps} to the smoothing parameter $\epsilon_k$.
Then $u_k \to \bar u$ in $V$.
\end{lemma}
\begin{proof}
The first part of the proof is similar to the one of \cite[Proposition 5.3]{ItoKunisch2014}.
By construction, $(u_k)$ is bounded in $V$
so that (after extracting a subsequence if necessary) we have
$u_k \rightharpoonup u^*$ in $V$ and $u_k\to u^*$ in $L^2(\Omega)$.
Then $G_{\epsilon_k}(u_k) \to G_0(u^*)$ and $G_{\epsilon_k}(\bar u) \to G_0(\bar u)$ by \cref{lem_conv_psi_eps},
which implies $\Phi_{\epsilon_k}(\bar u) \to  \Phi_0(\bar u) $ and $\liminf \Phi_{\epsilon_k}(u_k) \ge \Phi_0( u^*)$.
Taking the limit inferior in
\[
\Phi_{\epsilon_k}(\bar u) \ge \Phi_{\epsilon_k}(u_k) + \frac12 \|u_k-\bar u\|_{{\color{red} L^2(\Omega)}}^2
\]
leads to
\[
 \Phi_0(\bar u)  \ge \Phi_0( u^*) + \frac12 \|u^*-\bar u\|_{{\color{red} L^2(\Omega)}}^2
 \ge \Phi_0(\bar u),
\]
where the latter inequality follows from local optimality of $\bar u$. This proves $u^*=\bar u$.
From elementary properties of limit inferior and superior, we get
\[\begin{split}
  \Phi_0(\bar u)
  &\ge \limsup_{k\to\infty} \left( F(u_k) + \frac\alpha2 \|u_k\|_{V}^2 + \beta G_{\epsilon_k}(u_k)+ \frac12 \|u_k-\bar u\|_{{\color{red} L^2(\Omega)}}^2 \right) \\
  &\ge \liminf_{k\to\infty} \left( F(u_k) + \beta G_{\epsilon_k}(u_k) \right)+ { \color{red}\limsup_{k\to\infty}  \frac\alpha2 \|u_k\|_{V}^2 +  0 }\\
  &\ge \liminf_{k\to\infty} \left( F(u_k) + \beta G_{\epsilon_k}(u_k) \right)+ { \color{red}\liminf_{k\to\infty}  \frac\alpha2 \|u_k\|_{V}^2  }\\
  &\ge   \Phi_0(\bar u).
\end{split}\]
Hence,  {\color{red} $\|u_k\|_{V}^2\to\|\bar u\|_{V}^2$ and }$u_k\to \bar u$ in $V$.
Since all subsequences of $(u_k)$ contain subsequences converging strongly to $\bar u$, the convergence of the whole sequence follows.
\end{proof}

A particular implication of the previous result is that the constraint $\|u_\epsilon-\bar u \|_{V} \le \rho$ will be satisfied
as strict inequality if $\epsilon$ is sufficiently small.

\begin{lemma}\label{lem_optcond_Pauxrho}
 Let $u_\epsilon$ be a global solution of the auxiliary problem \eqref{eq_aux_Phi_eps} to $\epsilon>0$ with $\|u_\epsilon-\bar u\|_{V}<\rho$.
 Then $u_\epsilon$ satisfies
 \begin{equation}\label{eq_optcon_Pauxrho}
  \alpha \langle u_\epsilon, v\rangle_V  + \beta G_\epsilon'(u_\epsilon)v  + \langle u_\epsilon-\bar u, \, v\rangle_{{\color{red} L^2(\Omega)}}  = - F'(u_\epsilon)v \qquad \forall v\in V.
 \end{equation}
\end{lemma}
\begin{proof}
Due to the assumptions, $u_\epsilon$ is a local minimum of $u\mapsto \Phi_\epsilon(u) + \frac12 \|u-\bar u\|_{{\color{red} L^2(\Omega)}}^2$. This functional is differentiable according to \cref{ass_standing} and
\cref{lem_diff_Geps}.
\end{proof}

We will now pass to the limit $\epsilon\searrow0$ in \eqref{eq_optcon_Pauxrho}.
The term $\langle u_\epsilon-\bar u, \, v\rangle_{{\color{red} L^2(\Omega)}}$ will disappear according to \cref{lem_conv_aux_solutions}.
Thus, the next step is to investigate the behavior of $G_\epsilon'(u_\epsilon)$ for $\epsilon\searrow0$.

\begin{lemma}\label{lem_convergence_lambdak}
Let $(\epsilon_k)$ be a sequence of positive numbers with $\epsilon_k\to 0$, and $(u_k)$ be a sequence of global solutions of \eqref{eq_aux_Phi_eps}.
Define
\[
\lambda_k := G_{\epsilon_k}'(u_k).
\]
Then $\lambda_k \to \bar\lambda$  in $V^*$.
\end{lemma}
\begin{proof}
This is a consequence of the convergence $u_k \to \bar u$ by \cref{lem_conv_aux_solutions} and the continuity of $F'$.
\end{proof}

\begin{lemma}\label{lem_est_Gpu}
Let $u\in V$ and $\epsilon>0$ be given. Then
\[
  p \int_\Omega |u|^p \dx
 \ge
  G_\epsilon'(u^2)u
  =
  p \int_\Omega \min(\epsilon^{p-2},|u|^{p-2})|u|^2 \dx
  .
 \]
\end{lemma}
\begin{proof}
The integrand in the expression $G_\epsilon'(u^2)u$ according to \cref{lem_diff_Geps} is of the type
\[
 2\psi'_{\epsilon}(t^2)t^2 = p \min(\epsilon^{p-2},|t|^{p-2})t^2 \qquad \forall t\in \R,
\]
which immediately implies
\[
 p  |t|^p \ge 2\psi'_{\epsilon}(t^2)t^2  =  p \min(\epsilon^{p-2},|t|^{p-2})t^2 \qquad \forall t\in \R,
\]
and the claims follows by integration.
\end{proof}

\begin{lemma}\label{lem_conv_lambdau}
Let $(\epsilon_k),(u_k),(\lambda_k)$ be as in \cref{lem_convergence_lambdak} with $u_k\to \bar u$ and $\lambda_k\to \bar \lambda$ in $V$ and $V^*$, respectively.
Then
\[
 \langle \lambda_k, u_k\rangle_{V^*,V} \to \langle \bar\lambda, \bar u\rangle_{V^*,V}  =  p \int_\Omega |\bar u|^p \dx.
 \]
\end{lemma}
\begin{proof}
Let us recall the definition $\lambda_k = G_{\epsilon_k}'(u_k)$, so that \cref{lem_est_Gpu} gives an upper and lower bound of $\langle \lambda_k, u_k\rangle_{V^*,V}$.
We will show that both bounds converge to the same value.
 Since $u_k\to\bar u$ in $V$ we can assume (after possibly extracting a subsequence) that $u_k \to \bar u$ pointwise a.e.,
 with $|u_k|\le w$  pointwise a.e.\@ for some $w\in L^2(\Omega)$, see \cite[Thm 4.9]{Brezis2011}.
 Hence, $p w^p \in L^1(\Omega)$ is an integrable pointwise upper bound of the integrands in both integrals of \cref{lem_est_Gpu} evaluated at $u_k$.
Clearly,  $\int_\Omega |u_k|^p \dx \to \int_\Omega |\bar u|^p \dx$ by dominated convergence.
For the second integral, we have
\[
 \int_\Omega \min(\epsilon_k^{p-2},|u_k|^{p-2})|u_k|^2 \dx \ge \int_{\{x:\bar u(x)\ne0\}} \min(\epsilon_k^{p-2},|u_k|^{p-2})|u_k|^2 \dx.
\]
If $x\in \Omega$ is such that $\bar u(x)\ne0$ then we get $\min(\epsilon_k^{p-2},|u_k(x)|^{p-2}) \to |\bar u(x)|^{p-2}$.
And we can use dominated convergence theorem to pass to the limit $k\to\infty$ as follows
\[
\int_{\{x:\bar u(x)\ne0\}} \min(\epsilon_k^{p-2},|u_k|^{p-2})|u_k|^2 \dx
 \to \int_{\{x:\bar u(x)\ne0\}} |\bar u|^p \dx = \int_\Omega |\bar u|^p \dx.
\]
The claim follows now with \cref{lem_est_Gpu}.
\end{proof}

In this sense, $\bar \lambda$ can be interpreted as the derivative of $u\mapsto |u|^p$ at $\bar u$.

\begin{remark}
Let us assume that $u\zeta \in V$ for all $u\in V$ and $\zeta\in C_c^\infty(\Omega)$.
Then one can prove with similar arguments as in \cref{lem_conv_lambdau}
that
 \[
  \langle \bar\lambda,\zeta \bar u\rangle_{V^*,V}  =  p \int_\Omega |\bar u|^p \zeta \dx
 \]
for all $\zeta\in C_c^\infty(\bar \Omega)$ with $\zeta\ge0$.
Hence, $ \bar\lambda \bar u$ can be interpreted as a non-negative distribution.
\end{remark}

Now we have everything at hand to be able to prove optimality conditions for the original problem \eqref{eq001}.

\begin{theorem}\label{thm_opt_cond}
 Let $\bar u$ be a local solution of the original problem. Then there is $\bar\lambda\in V^*$ such that
 \begin{equation}\label{eq_optcon_P}
  \alpha \langle \bar u, v\rangle_V  + \beta \langle \bar\lambda, v\rangle_{V^*,V}   = - F'(\bar u)v \qquad \forall v\in V.
 \end{equation}
 and
\[
   \langle \bar\lambda, \bar u\rangle_{V^*,V}  =  p \int_\Omega |\bar u|^p \dx.
 \]
\end{theorem}
\begin{proof}
This is a consequence of the results above.
\end{proof}

\begin{remark}
Unfortunately and interestingly,
 the origin $\bar u=0$ solves the system of \cref{thm_opt_cond} with the trivial choice $\bar\lambda=-\beta^{-1} F'(0)$.
 This is in line with the observation of \cref{rem_origin_local_sol} that the origin is a minimum along lines.
\end{remark}

Under stronger assumptions on $F$ and on the underlying space $V$, we can prove additional regularity of $\bar u$ and $\bar\lambda$.

\begin{assumption}[Truncation and density] \phantom{\quad}
\label{ass_truncation}
\begin{enumerate}
	\item
	\label{ass_item_trunc}
For all $u\in V$, we have that $v$ defined by $v(x):=\max(-1, \min(u(x),+1))$ belongs to $V$ and satisfies $\langle u,v\rangle_V \ge \|v\|_V^2$.
\item
	\label{ass_item_density}
	$C_0(\Omega)\cap V$ is dense in $V$ and $C_0(\Omega)$ with respect to their norms.
\end{enumerate}
\end{assumption}

% Under this assumption, arbitrary truncations of functions from $V$ belong to $V$: in fact, for $u\in V$ and $a<b$ we have
% \[
% 	\max(a, \min(u,b)) = \frac{b-a}2 \max\left(-1,\ \min\left( \frac{u-\frac{b+a}2}{\frac{b-a}2}, +1\right)\right) + \frac{b+a}2 \in V.
% \]

\begin{lemma}\label{lem_boundedl1_lambda}
Suppose \cref{ass_truncation} is satisfied.
 Let $u_\epsilon$ be a solution of \eqref{eq_optcon_Pauxrho} such that $F'(u_\epsilon)\in L^1(\Omega)$.
 Define $\lambda_\epsilon:=G'(u_\epsilon)$. Then
 \[
\beta \|\lambda_\epsilon\|_{L^1(\Omega)} \le \|F'(u_\epsilon)\|_{L^1(\Omega)} + \|\bar u\|_{L^1(\Omega)}.
 \]
 \end{lemma}
\begin{proof}
We follow an idea of \cite[Theorem 5.1]{ItoKunisch2000}.
Let us test \eqref{eq_optcon_Pauxrho} with $v_n:=\max(-1,\min(n\cdot u_\epsilon,+1)) \approx \operatorname{sign}(u_\epsilon)$, where $n\in \mathbb N$.
Using \cref{ass_truncation}, we get $v_n\in V$ and $\langle u,v_n\rangle_V \ge n^{-1} \|v_n\|_V^2 \ge0$.
{\color{red}
In addition, $\langle u_\epsilon, \, v_n\rangle_{L^2(\Omega)}\ge 0$ and $|\langle \bar u, \, v_n\rangle_{L^2(\Omega)}| \le \|\bar u\|_{L^1(\Omega)}$.
}
With $v_n$ as test function in  \eqref{eq_optcon_Pauxrho}, we get
\[
{\color{red}\alpha\, } n^{-1} \|v_n\|_V^2 + \beta \int_\Omega \lambda_\epsilon v_n\dx  \le \|F'(u_\epsilon)\|_{L^1(\Omega)} + \|\bar u\|_{L^1(\Omega)}.
\]
The integral involving $\lambda_\epsilon$ can be written as
\[
\int_\Omega \lambda_\epsilon v_n\dx = \int_\Omega 2u_\epsilon(x)\cdot \psi_\epsilon'(u_\epsilon(x)^2) \cdot v_n(x)\ \dx,
\]
see \cref{lem_diff_Geps}. We can pass to the limit $n\to\infty$ by dominated convergence to obtain
\begin{multline*}
\int_\Omega \lambda_\epsilon v_n\dx = \int_\Omega 2u_\epsilon(x)\cdot \psi_\epsilon'(u_\epsilon(x)^2) \cdot v_n(x)\ \dx
\\
\to  \int_\Omega 2|u_\epsilon(x)|\cdot \psi_\epsilon'(u_\epsilon(x)^2)  \dx =\|\lambda_\epsilon\|_{L^1(\Omega)},
\end{multline*}
which is the claim.
\end{proof}

\begin{theorem}
Suppose \cref{ass_truncation} is satisfied.
Assume $F'$ is a continuous mapping from $V$ to $L^1(\Omega)$.
 Let $\bar u$ be a local solution of the original problem \eqref{eq_problem}.
 Then $\bar\lambda\in V^*$ as in \cref{thm_opt_cond}
 belongs to $\mathcal M(\Omega) =C_0(\Omega)^*$.
\end{theorem}
\begin{proof}
Due to the assumptions, $(F'(u_k))$ is bounded in $L^1(\Omega)$.
By \cref{lem_boundedl1_lambda}, the sequence  $(\lambda_k)$ is bounded in $L^1(\Omega)$.
We can identify $L^1(\Omega)$ with a subspace of $C_0(\Omega)^*$ using the natural embedding.
Hence (after extracting a subsequence if necessary), we have $\lambda_k \rightharpoonup^* \tilde\lambda$ in $C_0(\Omega)^*$.
Due to the density assumption in \cref{ass_truncation}, it follows $\tilde\lambda = \bar\lambda$.
\end{proof}

In order to obtain $L^\infty$-regularity, we need the following embedding assumption.

\begin{assumption}[$L^q$-embedding]
\label{ass_lq_embedding}
There is $q>2$ such that $V$ is continuously embedded in $L^q(\Omega)$.
\end{assumption}

The following lemma mimics one key step in the proof of the celebrated $L^\infty$-regularity result for weak solutions of elliptic partial differential
equations of \cite{Stampacchia1965}.

\begin{lemma}[$L^\infty$-regularity]
\label{lem_linfty}
Let \cref{ass_truncation} and \cref{ass_lq_embedding} be satisfied. Let $t>1$ be such that $\frac1t + \frac2q < 1$.
Then there is $c>0$ such that for all $g\in L^t(\Omega)$ and $u\in V$ the property
	\[
		\langle u,v_n\rangle_V  \le \int_\Omega g v_n\dx \quad \forall n\in \N
	\]
	with $v_n:=u-\max(-n,\min(u,+n))$
	implies
	\[
		\|u\|_{L^\infty(\Omega)} \le c \|g\|_{L^t(\Omega)}.
	\]
\end{lemma}
\begin{proof}
This is \cite[Lemma 3.4]{ChristofGW} with $k_0=0$ and $\sigma:=(\frac1t + \frac2q -1)^{-1}<0$.
\end{proof}

% \begin{lemma}\label{lem_boundedlinfty_u}
% Suppose \cref{ass_linfty} is satisfied.
%  Let $u_\epsilon$ be a global solution of the auxiliary problem \eqref{eq_aux_Phi_eps} to $\epsilon>0$ with $\|u_\epsilon-\bar u\|_{V}<\rho$.
%  Assume $F'(u_\epsilon) \in L^q(\Omega)$ and $\bar u\in L^q(\Omega)$. Then
%  \[
%  \|u_\epsilon\|_{L^\infty(\Omega)} \le \frac{c_{q,\infty}}{\alpha+1}( \|F'(u_\epsilon)\|_{L^q(\Omega)} + \|\bar u\|_{L^q(\Omega)} ),
%  \]
%  where $q$ and $c_{q,\infty}$ are as in \cref{ass_linfty}.
% \end{lemma}
% \begin{proof}
% Testing equation \eqref{eq_optcon_Pauxrho} with $v_n:=\max(-n,\min(u_\epsilon,+n))$ results in
% \[
% (\alpha+1)\langle u,v_n\rangle_V \le \int_\Omega (\bar u-F'(u_\epsilon)) v_n\dx \quad \forall n\in \N,
% \]
% where we have used the specific form of $G_\epsilon'$ given in \cref{lem_diff_Geps}. The claim follows by \cref{ass_linfty}.
% \end{proof}

\begin{theorem}\label{thm_linfty_u}
Suppose  \cref{ass_truncation} and \cref{ass_lq_embedding} are satisfied.
 Let $\bar u$ be a local solution of the original problem.
 Assume $F'(\bar u) \in L^s(\Omega)$ with $s$ as in \cref{lem_linfty}. Then
 $\bar u\in L^\infty(\Omega)$.
\end{theorem}
\begin{proof}
Let us set $v_n:=\bar u - \max(-n,\min(\bar u,+n))$.
Let $(\epsilon_k),(u_k),(\lambda_k)$ be as in \cref{lem_convergence_lambdak} with $u_k\to \bar u$ and $\lambda_k\to \bar \lambda$ in $V$ and $V^*$, respectively.
Then arguing as in the proof of \cref{lem_conv_lambdau} one can show
\[\begin{split}
\langle \lambda_k, v_n \rangle_{V^*,V}
 &=  \int_\Omega \min(\epsilon_k^{p-2},|u_k|^{p-2})u_k v_n \dx
 \\
 &= \int_{\{x:\bar u(x)\ne0\}} \min(\epsilon_k^{p-2},|u_k|^{p-2})u_kv_n \dx
  \\
 & \to \int_{\{x:\bar u(x)\ne0\}} |\bar u|^{p-1}\bar uv_n \dx\ge 0
\end{split}\]
for $k\to\infty$
with the help of dominated convergence.
Testing \eqref{eq_optcon_Pauxrho} for $(u_k,\epsilon_k)$ instead of $(u_\epsilon,\epsilon)$ with $v_n$ and passing to the limit $k\to\infty$, gives the inequality
\[
	\alpha\langle \bar u,v_n \rangle_V  \le -F'(\bar u)v_n \quad\forall n\in \N.
\]
Now the claim follows from \cref{lem_linfty}.
\end{proof}

\begin{remark}
	If $F':V\to L^s(\Omega)$ is continuous then one can prove the boundedness of $(u_k)$ in $L^\infty(\Omega)$
	with similar arguments as in the proof of \cref{thm_linfty_u} above.
\end{remark}

\section{Discussion of assumptions}
\label{sec_fractional}

\cref{ass_standing}, \cref{ass_truncation}, and \cref{ass_lq_embedding} are satisfied when $V=H^1(\Omega)$ or when $V$ is a
fractional order Sobolev space. The former is straightforward. Next, we elaborate on the fractional
order Sobolev space setting.

Let $s \in (0,1)$,  and let $\Omega\subset \R^d$ be a Lipschitz domain.
Then the Sobolev space $H^s(\Omega)$ is defined by
\[
	H^s(\Omega) := \left\{ u \in L^2(\Omega): \ \int_{\Omega}\int_{\Omega} \frac{|u(x) - u(y)|^2}{|x-y|^{d+2s}} \dy \dx < +\infty \right\}
\]
with norm
\[
	\|u\|_{H^s(\Omega)}^2 :=  \|u\|_{L^2(\Omega)}^2 +  \int_{\Omega}\int_{\Omega} \frac{|u(x) - u(y)|^2}{|x-y|^{d+2s}} \dy \dx.
\]
Let us introduce in addition
\[
	H^s_0(\Omega) := \overline{C_c^\infty(\Omega)}^{H^s(\Omega)}
\]
and
\[
	\widetilde{H}^s(\Omega) := \left\{ w \in H^s(\mathbb{R}^d) \, : \, w|_{\mathbb{R}^d \setminus \Omega} = 0 \right\} =  \overline{C_c^\infty(\Omega)}^{H^s(\R^d)},
\]
where the latter identity is due to \cite[Theorem 6]{AFiscella_FServadei_EValdinoci_2015a}.
Here, functions from $C_c^\infty(\Omega)$ are tacitly extended to $\R^d$ by zero.
This  results in the inclusions with continuous embeddings
\[
 \widetilde{H}^s(\Omega)  \subseteq H^s_0(\Omega)  \subseteq H^s(\Omega),
\]
where in the first inclusion function in $\widetilde{H}^s(\Omega)$ are understood as functions defined on $\Omega$.
In addition, we have
 $\widetilde{H}^s(\Omega)  = H^s_0(\Omega)$ if $s\ne 1/2$, and $ \widetilde{H}^s(\Omega)  = H^s_0(\Omega)  = H^s(\Omega)$ if $s<1/2$,
see \cite[Corollary 1.4.4.5]{Grisvard1985}.
We further define the Lions-Magenes space
\[
	H^{\frac12}_{00}(\Omega)
	:= \left\{ u \in H^{\frac12}(\Omega) \, : \, \int_{\Omega} \frac{u^2(x)}{\mbox{dist}(x,\partial\Omega)} \dx < +\infty
	   \right\},
\]
with norm
\[
	\|u\|_{H^{\frac12}_{00}(\Omega)} = \left( \|u\|^2_{H^{\frac12}(\Omega)}
		+  \int_{\Omega} \frac{u^2(x)}{\mbox{dist}(x,\partial\Omega)} \dx  \right)^{\frac12} .
\]
Now, we can prove that parts of the assumptions are satisfied for fractional Sobolev spaces.
Note that these parts of the assumptions only depend on the properties of the Hilbert spaces but not on the concrete choice of the inner product.

\begin{lemma}
Let $s \in (0,1)$,  and let $\Omega\subset \R^d$ be a Lipschitz domain.
 Let $V$ be one of the spaces $\widetilde{H}^s(\Omega)$, $H^s_0(\Omega)$, $H^{\frac12}_{00}(\Omega)$. %, $H^s(\Omega)$.
 Then the following conditions are satisfied:
 \begin{enumerate}
 \renewcommand{\theenumi}{(\alph{enumi})}
  \item\label{it_lem51_1} \cref{ass_standing}-\ref{ass1_V},
  \item\label{it_lem51_3} \cref{ass_truncation}-\ref{ass_item_density},
  \item\label{it_lem51_2} \cref{ass_lq_embedding}.
\end{enumerate}
  \end{lemma}
\begin{proof}
 Clearly, the embedding $\widetilde{H}^s(\Omega) \hookrightarrow H^s(\Omega)$ is continuous.
 The embedding $H^s(\Omega) \hookrightarrow L^2(\Omega)$ is compact \cite[Theorem 7.1]{DDiNezza_GPalatucci_EValdinoci_2012a}.
 Since $V$ contains $C_c^\infty(\Omega)$, $V$ is dense in $L^2(\Omega)$, which proves \ref{it_lem51_1}.
 Moreover, $V\cap C_c^\infty(\Omega) = C_c^\infty(\Omega)$, which is dense in $V$ and $C_0(\Omega)$.
 Here, the density of $C_c^\infty(\Omega)$ in $H^{\frac12}_{00}(\Omega)$ is a consequence of the representation of this space by interpolation $H^{\frac12}_{00}(\Omega) = [H^1_0(\Omega), L^2(\Omega)]_{1/2}$,
 see \cite[Theorem 11.7]{LionsMagenes1972},
 and of the density of $C_c^\infty(\Omega)$ in $H^1_0(\Omega) = H^1_0(\Omega)\cap L^2(\Omega)$, see \cite[Theorem 4.2.2]{BerghLofstrom1976}.
 This proves \ref{it_lem51_3}.
 In addition, there is $q>2$ such that the embedding $H^s(\Omega) \hookrightarrow L^q(\Omega)$ is continuous \cite[Theorem 6.7]{DDiNezza_GPalatucci_EValdinoci_2012a},
 which is \ref{it_lem51_2}.
\end{proof}

It remains to check \cref{ass_truncation}-\ref{ass_item_trunc}, which is an assumption not only on the space $V$ but also on its inner product.
Here, we want to work with inner products induced by fractional Laplacians.

We consider two well-known definitions of fractional Laplacian \cite{HAntil_SBartels_ASchikorra_2021a,HAntil_JPfefferer_SRogovs_2018a,DDiNezza_GPalatucci_EValdinoci_2012a}.
We start with the integral fractional Laplacian. To define the integral fractional Laplace operator
we consider the weighted Lebesgue space
\[
	\mathbb{L}^1_s(\R^d) = \left\{ u : \R^d \rightarrow \R \mbox{ measurable and } \int_{\R^d}
		\frac{|u(x)|}{(1+|x|)^{d+2s}} \dx < \infty \right\} .
\]
For $u \in \mathbb{L}^1_s(\R^d)$, $\varepsilon > 0$, and $x \in \R^d$ we set
\[
	(-\Delta)^s_\varepsilon u(x) = C_{d,s} \int_{\{ y \in \R^d , \ |y-x| > \varepsilon \}}
		\frac{u(x)-u(y)}{|x-y|^{d+2s}} \dy ,
\]
where $C_{d,s} = \frac{s2^{2s} \Gamma(s+\frac{d}{2})}{\pi^{\frac{d}{2}} \Gamma(1-s)}$ is
a normalization constant. Then the \emph{integral fractional Laplacian} is defined
for $s \in (0,1)$ by taking the limit $\varepsilon \rightarrow 0$, i.e.,
\begin{equation}\label{eq:integ}
	(-\Delta)^s u(x) = C_{d,s} \mbox{P.V.} \int_{\mathbb{R}^d} \frac{u(x) - u(y)}{|x-y|^{d+2s}} \dy
		= \lim_{\varepsilon \rightarrow 0} (-\Delta)^s_\varepsilon u(x) ,
\end{equation}
where P.V. denotes the Cauchy principal value.

Due to \cite[Proposition 3.6]{AFiscella_FServadei_EValdinoci_2015a},
an equivalent norm on $\widetilde{H}^s(\Omega)$ is given by $u \mapsto \| (-\Delta)^{\frac{s}{2}} u \|_{L^2(\R^d)}$,
which motivates the following choice of the inner product
\begin{equation}\label{eq_innerpr_tildeh}
 \langle u,v \rangle_{\widetilde{H}^s(\Omega)} := \int_{\R^d}\int_{\R^d} \frac{(u(x) - u(y))(v(x) - v(y))}{|x-y|^{d+2s}} \dy \dx.
\end{equation}

\begin{lemma}\label{lem_trunc_tildeh}
 Let $V:=\widetilde{H}^s(\Omega)$ provided with the inner product \eqref{eq_innerpr_tildeh}. Then \cref{ass_truncation}-\ref{ass_item_trunc}
 is satisfied.
\end{lemma}
\begin{proof}
 Let $u \in \widetilde{H}^s(\Omega)$. Then the truncated function $v:=\max(-1, \min(u,+1))$ belongs to $H^s(\R^d)$ by \cite[Lemma 2.7]{Warma2015},
 and consequently $v\in \widetilde{H}^s(\Omega)$.
 The inequality $\langle u,v\rangle_{\widetilde{H}^s(\Omega)} \ge \|v\|_{\widetilde{H}^s(\Omega)}^2$ can be proven following
 \cite[Proof of Theorem 2.9, (2.30)]{AntilPfeffererWarma2017}.
\end{proof}

Next, we discuss the spectral definition. Let $-\Delta_\Omega$ be the realization in $L^2(\Omega)$ of the Laplace
operator with zero Dirichlet boundary conditions. By classical results, $-\Delta_\Omega$ has a compact resolvent
and its eigenvalues form a non-decreasing sequence $0 < \mu_1 \le \mu_2 \le \cdots \le \mu_k \le \cdots$
with $\lim_{k\rightarrow \infty}
\mu_k = \infty$. Let $\psi_k \in H^1_0(\Omega)$ be the orthonormal eigenfunctions associated with
$\mu_k$. These eigenfunctions form an orthonormal basis of $L^2(\Omega)$. Then for any
$u\in C_c^\infty(\Omega)$, the fractional powers of $-\Delta_\Omega$ can be defined as
\begin{equation}\label{eq:spect}
	(-\Delta_\Omega)^s u := \sum_{k=1}^\infty \mu_k^s u_k \psi_k \quad \mbox{with} \quad
		u_k = \int_\Omega u \psi_k \dx .
\end{equation}
By density, this definition can be extended to the Sobolev space
\[
	\mathbb{H}^s(\Omega) := \left\{ u \in L^2(\Omega) \, : \,
			\| u \|_{\mathbb{H}^s(\Omega)}^2 = \| (-\Delta_\Omega)^{\frac{s}{2}} u \|^2_{L^2(\Omega)} := \sum_{k=1}^\infty \mu_k^s u_k^2 < \infty \right\} .
\]
% For the definition of $(-\Delta_\Omega)^s$ with nonhomogeneous boundary conditions, see
% \cite{HAntil_JPfefferer_SRogovs_2018a}.
% We denote the dual space of $\mathbb{H}^s(\Omega)$ by $\mathbb{H}^{-s}(\Omega)$.
The space
$\mathbb{H}^s(\Omega)$ can be characterized in terms of the fractional order Sobolev space defined
above.
Following \cite[Chapter 1]{LionsMagenes1972}, we have that
\[
	\mathbb{H}^s(\Omega)
	= \left\{
	   \begin{array}{ll}
	   	H^s_0(\Omega) & \mbox{if } s \neq \frac12 \\
		H^{\frac12}_{00}(\Omega) & \mbox{if } s = \frac12 .
	   \end{array}
	  \right.
\]

\begin{lemma}
 Let $V:=\mathbb{H}^s(\Omega)$ provided with the inner product
\[
 \langle u,v \rangle_{\mathbb{H}^s(\Omega)} := \langle (-\Delta_\Omega)^{\frac{s}{2}} u ,\ (-\Delta_\Omega)^{\frac{s}{2}}v\rangle_{L^2(\Omega)}.
\]
Then \cref{ass_truncation}-\ref{ass_item_trunc}
 is satisfied.
\end{lemma}
\begin{proof}
The proof is essentially the same as the one of \cref{lem_trunc_tildeh}.
 Let $u \in \mathbb{H}^s(\Omega)$.
If $s \ne 1/2$ then $ \mathbb{H}^s(\Omega) = \widetilde{H}^s(\Omega)$ and $v:=\max(-1, \min(u,+1))$ belongs to $\mathbb{H}^s(\Omega)$ by \cref{lem_trunc_tildeh}.
If $s=1/2$ then $v\in H^{1/2}(\Omega)$ by  \cite[Lemma 2.7]{Warma2015}, and $v\in H^{\frac12}_{00}(\Omega)$ follows.
 The inequality $\langle u,v\rangle_{\mathbb{H}^s(\Omega)} \ge \|v\|_{\mathbb{H}^s(\Omega)}^2$ is equivalent to
 \cite[(2.16)]{AntilPfeffererWarma2017}.
\end{proof}

\section{Iterative scheme}
\label{sec_algorithm}

Throughout this section, we suppose that $F$ fulfills the following condition.

\begin{assumption} \label{ass_iteration}
	Let \cref{ass_standing} be satisfied. In addition, we require:
	\begin{enumerate}
% 		\item $F': V \to V^*$ is locally Lipschitz continuous,
		\item $F'$ is completely continuous, i.e., $u_n \rightharpoonup u$ in $V$ implies $F'(u_n)\to F'(u)$ in $V^*$ for all sequences $(u_n)$.
		\item $F': V \to V^*$ is Lipschitz continuous on bounded sets, i.e., for all $R>0$ there is $L>0$ such that
		$\|F'(u) - F'(v)\|_{V^*} \le L\|u-v\|_V$ for all $u,v$ with $\|u\|_V,\|v\|_V\le R$.
		% see Zeidler volume II, section 42.6
	\end{enumerate}
\end{assumption}

We will use the following algorithm to compute candidates of solutions for the optimization problem.
Similars method were used in \cite{ItoKunisch2014}, where $F$ was assumed to be quadratic, and \cite{GeipingMoeller2018},
where a more abstract but finite-dimensional problem was analyzed.

% \alert{Similar to \cite{GeipingMoeller2018}!, smoothing appears in \cite{SongBabuPalomar2015} and references }

\begin{algorithm}\label{alg}
\begin{enumerate}
\item Choose monotonically decreasing sequence $(\epsilon_k)$ with $\epsilon_k\searrow0$, $\beta>1$, $\tilde L>0$.
	\item
	\label{step2}
	Determine $L_k\ge0$ to be the smallest number from $\{0\} \cup \{ \tilde L \beta^l, l \ge0\}$
	for which the solution $u_{k+1}$ of
	\begin{multline}\label{eq_iter_opt}
		\min_{u\in V} F(u_k) + F'(u_k)(u-u_k) + \frac {L_k}2 \|u-u_k\|_V^2 +\frac\alpha2\|u\|_V^2 \\
		+ \beta \int_\Omega \psi_{\epsilon_k}(u_k^2) + \psi_{\epsilon_k}'(u_k^2)(u^2-u_k^2)\dx,
	\end{multline}
	satisfies
	\begin{equation}\label{eq_iter_descent}
		F(u_{k+1}) \le  F(u_k)+  F'(u_k)(u_{k+1}-u_k) + L_k \|u_{k+1}-u_k\|_V^2.
	\end{equation}
\item
Set $k:=k+1$ and go to step \ref{step2}.
\end{enumerate}

\end{algorithm}

The optimization problem in \eqref{eq_iter_opt} is strongly convex since $\alpha>0$ and $\psi_{\epsilon_k}'\ge0$.
Hence, \eqref{eq_iter_opt} admits a unique solution for each $u_k$ and $L_k>0$.
In the following, we want to prove that the sequence $(u_k)$ is bounded in $V$. In addition, we are interested in proving the weak limit
points satisfy conditions similar to the one derived in \cref{thm_opt_cond}.

First, let us argue that we can find $L_k$ and $u_{k+1}$ satisfying the descent condition \eqref{eq_iter_descent}.
This is a consequence of the local Lipschitz continuity of $F'$.
If $F'$ is (global) Lipschitz continuous then \eqref{eq_iter_descent} is fulfilled as soon as $L_k$ is larger than the Lipschitz modulus of $F'$,
which is implied by the so-called descent lemma.

\begin{lemma}[Descent lemma]
\label{lem_descent}
 Let $M>0$ and $\rho>0$ be such that $\|F'(u)-F'(v)\|_{V^*} \le M \|u-v\|_V$ for all $v\in B_\rho(u)$.
 Then
 \[
  F(v) \le F(u) + F'(u)(v-u) + \frac M2 \|v-u\|_V^2 \quad \forall v\in B_\rho(u).
 \]
\end{lemma}
\begin{proof}
This is a consequence of the mean-value theorem for G\^ateaux differentiable functions, see, e.g., \cite[Proposition 3.3.4]{Schirotzek2007}.
\end{proof}

\begin{lemma}\label{lem_Lk}
Under \cref{ass_iteration}, for each $u_k$ there exists $L_k\ge0$ and $u_{k+1}\in V$ satisfying the conditions \eqref{eq_iter_opt}--\eqref{eq_iter_descent}.

In addition, if $(u_k)$ is bounded in $V$ then $(L_k)$ is bounded.
\end{lemma}
\begin{proof}
For $n\in \N$, let $w_{n,k}$ be the solution of
\begin{multline*}
    \min_{u\in V} F(u_k) + F'(u_k)(u-u_k) + \frac n2 \|u-u_k\|_V^2 +\frac\alpha2\|u\|_V^2+ \\
    \beta \int_\Omega \psi_{\epsilon_k}(u_k^2) + \psi_{\epsilon_k}'(u_k^2)(u^2-u_k^2)\dx.
\end{multline*}
By optimality of $w_{n,k}$, we have
\begin{multline*}
F(u_k)  +\frac\alpha2\|u_k\|_V^2+ \beta \int_\Omega \psi_{\epsilon_k}(u_k^2) \dx
\\
\ge
	F(u_k) + F'(u_k)(w_{n,k}-u_k) + \frac n2 \|w_{n,k}-u_k\|_V^2 +\frac\alpha2\|w_{n,k}\|_V^2\\
	+ \beta \int_\Omega \psi_{\epsilon_k}(u_k^2) + \psi_{\epsilon_k}'(u_k^2)(w_{n,k}^2-u_k^2)\dx.
\end{multline*}
Due to the concavity and non-negativity of $\psi_{\epsilon_k}$, we obtain $\psi_{\epsilon_k}(u_k^2) + \psi_{\epsilon_k}'(u_k^2)(w_{n,k}^2-u_k^2) \ge \psi_{\epsilon_k}(w_{n,k}^2) \ge0$.
This implies
\begin{equation}
\label{eq613}
\begin{split}
\frac\alpha2\|u_k\|_V^2+  \beta \int_\Omega \psi_{\epsilon_k}(u_k^2)
&\ge
	F'(u_k)(w_{n,k}-u_k) + \frac n2 \|w_{n,k}-u_k\|_V^2  \\
	&\ge \frac n4  \|w_{n,k}-u_k\|_V^2 - \frac 1n \|F'(u_k)\|_{V^*}^2
\end{split}
\end{equation}
Consequently,  $w_{n,k}\to u_k$ for $n\to \infty$. By local Lipschitz continuity of $F'$ and \cref{lem_descent},
there are $M>0$ and $N>0$ such that
% $\|F'(u_k) - F'(w_{n,k})\|_{V^*} \le M \|u_k -w_{n,k}\|_V$
% for all $n>N$. Then the descent lemma implies
for all $n>N$
\[
	F(w_{n,k}) \le F(u_k) + F'(u_k)(w_{n,k}-u_k) + \frac M2\|w_{n,k}-u_k\|_V^2
\]
holds,
and condition \eqref{eq_iter_descent} is satisfied for $L_k>\max(N,\frac M2)$.

Let us now assume that the sequence $(u_k)$ is bounded, i.e, $\|u_k\|_V\le R$ for all $k$.
From \eqref{eq613} and the properties of $F'$, we find that there is $K>0$ such that $\|w_{n,k}-u_k\|_V^2 \le K/n$ for all $k$ and $n$.
Then $\|w_{n,k}\|_V \le R+1$ for all $n>K$ and all $k$.
Let $M$ be the Lipschitz modulus of $F'$ on $B_{R+1}(0)$. Then the  descent condition  \eqref{eq_iter_descent} is satisfied
whenever $L_k \ge \max(K, \frac M2)$. Due to the selection strategy of $L_k$ in \cref{alg}, it follows $L_k \le \max(K, \frac M2) \beta$.
\end{proof}

Using arguments as in the proof of \cref{lem_optcond_Pauxrho}, the iterate $u_{k+1}$ satisfies the following optimality condition
\begin{equation}\label{eq_optcon_iter}
\alpha \langle u_{k+1},v\rangle_V + L_k \langle u_{k+1}-u_k,v\rangle_V + F'(u_k)v + 2\beta \int_\Omega \psi_{\epsilon_k}'(u_k^2)u_{k+1}v\dx =0.
\end{equation}

Let us prove a first basic estimate, which gives us monotonicity of function values.
A similar result (for convex and quadratic $F$) can be found in \cite[Theorem 5.4]{ItoKunisch2014}.
Recall the definition of $\Phi_\epsilon$ in \eqref{eq_def_Phi_eps}, i.e.,
\[
	\Phi_\epsilon(u) := F(u) + \frac\alpha2 \|u\|_{V}^2 + \beta G_\epsilon(u).
\]

\begin{lemma}\label{lem_iter_basic}
	Let $(L_k,u_k)$ be  a sequence generated by \cref{alg} with $(\epsilon_k)$ monotonically decreasing.
	Then we have the following inequality
\[
\Phi_{\epsilon_{k+1}}(u_{k+1})
+\frac\alpha2  \|u_{k+1}-u_k\|_V^2
+\beta \int_\Omega \psi_{\epsilon_k}'(u_k^2)  (u_{k+1}-u_k)^2 \dx   %
\le
\Phi_{\epsilon_k}(u_k).
\]
\end{lemma}
\begin{proof}
Testing \eqref{eq_optcon_iter} with $u_{k+1}-u_k$ gives
\begin{multline*}
	\alpha\langle u_{k+1},u_{k+1}-u_k\rangle_V + L_k \|u_{k+1}-u_k\|_V^2 + F'(u_k)(u_{k+1}-u_k) \\
	+ 2\beta \int_\Omega \psi_{\epsilon_k}'(u_k^2)u_{k+1}(u_{k+1}-u_k)\dx =0
\end{multline*}
Using $a(a-b) = \frac12 ( (a-b)^2 + a^2 - b^2)$ to produce squares, we find
\begin{multline*}
\frac\alpha2 \left( \|u_{k+1}\|_V^2 + \|u_{k+1}-u_k\|_V^2 - \|u_k\|_V^2\right) + L_k \|u_{k+1}-u_k\|_V^2 + F'(u_k)(u_{k+1}-u_k) \\
+\beta \int_\Omega \psi_{\epsilon_k}'(u_k^2) ( u_{k+1}^2 - u_k^2 + (u_{k+1}-u_k)^2 )\dx =0,
\end{multline*}
which we can rearrange to
\begin{multline*}
\frac\alpha2 \left( \|u_{k+1}\|_V^2 + \|u_{k+1}-u_k\|_V^2 \right) + L_k \|u_{k+1}-u_k\|_V^2 +F(u_k)+  F'(u_k)(u_{k+1}-u_k) \\
+\beta \int_\Omega\psi_{\epsilon_k}(u_k^2) + \psi_{\epsilon_k}'(u_k^2) ( u_{k+1}^2 - u_k^2 + (u_{k+1}-u_k)^2 )\dx  \\
=\frac\alpha 2 \|u_k\|_V^2 + F(u_k) + \beta \int_\Omega\psi_{\epsilon_k}(u_k^2) \dx = \Phi_{\epsilon_k}(u_k).
\end{multline*}
Using condition \eqref{eq_iter_descent}, concavity of $t\mapsto \psi_{\epsilon_k}(t)$, and monotonicity of $(\epsilon_k)$ implies
\[\begin{split}
\Phi_{\epsilon_k}(u_k)
&\ge
\Phi_{\epsilon_k}(u_{k+1})
+\frac\alpha2  \|u_{k+1}-u_k\|_V^2
+\beta \int_\Omega \psi_{\epsilon_k}'(u_k^2)  (u_{k+1}-u_k)^2 \dx   %
\\
&\ge
\Phi_{\epsilon_{k+1}}(u_{k+1})
+\frac\alpha2  \|u_{k+1}-u_k\|_V^2
+\beta \int_\Omega \psi_{\epsilon_k}'(u_k^2)  (u_{k+1}-u_k)^2 \dx,
\end{split}\]
which is the claim.
\end{proof}

\begin{lemma}\label{lem_iter_bounded}
	Let $(L_k,u_k)$ be  a sequence generated by \cref{alg}.
	Then
	$(u_k)$ and $(F'(u_k))$ are bounded in $V$ and $V^*$, respectively.
\end{lemma}
\begin{proof}
By \cref{lem_iter_basic}, $(\Phi_{\epsilon_k}(u_k))$ is monotonically decreasing. Due to $\alpha>0$
and \cref{ass_standing}, $(u_k)$ is bounded in $V$.
Due to complete continuity of $F'$, cf., \cref{ass_iteration}, $(F'(u_k))$ is bounded in $V^*$.
\end{proof}

\begin{corollary}\label{coro_Lk_bounded}
	Let $(L_k,u_k)$ be  a sequence generated by \cref{alg}.
	Then $(L_k)$ is bounded.
\end{corollary}
\begin{proof}
Follows directly from \cref{lem_iter_bounded} and \cref{lem_Lk}.
\end{proof}

\begin{corollary}\label{coro_summing}
	Let $(L_k,u_k)$ be  a sequence generated by \cref{alg}.
	Then
    \[
\sum_{k=1}^\infty \frac\alpha2  \|u_{k+1}-u_k\|_V^2
+\beta \int_\Omega \psi_{\epsilon_k}'(u_k^2)  (u_{k+1}-u_k)^2 \dx  <\infty.
\]
\end{corollary}
\begin{proof}
 By \cref{ass_standing}, $\Phi_\epsilon$ is bounded from below uniformly in $\epsilon$.
 Summation of the inequality of \cref{lem_iter_basic} implies the claim.
\end{proof}

In order to be able to pass to the limit \eqref{eq_optcon_iter}, we need the following result.

\begin{lemma}\label{lem_iter_lambdau}
	Let $(L_k,u_k)$ be  a sequence generated by \cref{alg}.
	Let $\bar u$ be the weak limit of the subsequence $(u_{k_n})$ in $V$.
	Then
	\[
		2\int_\Omega \psi_{\epsilon_{k_n}}'(u_{k_n}^2)  u_{k_n+1}^2 \dx \to p \int_\Omega |\bar u|^p \dx.
	\]
% as in \cref{lem_conv_lambdau} - dominated converge should do the trick
\end{lemma}
\begin{proof}
 Due to \cref{coro_summing}, $u_{k_n+1} \rightharpoonup \bar u$ in $V$.
 After possibly extracting a subsequence if necessary, we can assume that $u_{k_n}$ and $u_{k_n+1}$ converge to $\bar u$  almost everywhere,
 and there is $w\in L^2(\Omega)$ such that $|u_{k_n}|, |u_{k_n+1}|\le w$ almost everywhere.
 Arguing as in the proof of \cref{lem_conv_lambdau}, we have
 \[
2\int_\Omega \psi_{\epsilon_{k_n}}'(u_{k_n}^2)  u_{k_n}^2 \dx = G_{\epsilon_k}'(u_k)u_k  \to p \int_\Omega |\bar u|^p \dx.
\]
\cref{coro_summing} implies
\[
 \int_\Omega \psi_{\epsilon_{k_n}}'(u_{k_n}^2)  (u_{k_n+1}-u_{k_n})^2 \dx  \to 0.
\]
Then by H\"older's inequality, we obtain
\begin{multline*}
\left| \int_\Omega \psi_{\epsilon_{k_n}}'(u_{k_n}^2) (u_{k_n+1}-u_{k_n}) u_{k_n} \dx \right|
\\
 \le
  \left( \int_\Omega \psi_{\epsilon_{k_n}}'(u_{k_n}^2)  (u_{k_n+1}-u_{k_n})^2 \right)^{1/2}
  \left( \int_\Omega \psi_{\epsilon_{k_n}}'(u_{k_n}^2)  u_{k_n}^2 \dx\right)^{1/2}
 \to 0,
\end{multline*}
which implies $2\int_\Omega \psi_{\epsilon_{k_n}}'(u_{k_n}^2)  u_{k_n+1}^2 \dx \to p \int_\Omega |\bar u|^p \dx$.
As the limit is independent of the chosen subsequence of $(u_{k_n})$, the claim is proven.
\end{proof}

\begin{theorem}\label{thm_optcon_lim_iter}
Let \cref{ass_iteration} be satisfied.
	Let $(L_k,u_k)$ be  a sequence generated by \cref{alg}.
	Let $\bar u$ be the weak limit in $V$ of a subsequence of $(u_k)$.
	Then there is $\bar\lambda\in V^*$ such that
 \begin{equation}\label{eq_optcon_lim_iter}
  \alpha \langle \bar u, v\rangle_V  + \beta \langle \bar\lambda, v\rangle_{V^*,V}    = - F'(\bar u)v \qquad \forall v\in V.
 \end{equation}
 and
\[
   \langle \bar\lambda, \bar u\rangle_{V^*,V}  \ge  p \int_\Omega |\bar u|^p \dx.
 \]
\end{theorem}
\begin{proof}
Let us define $\lambda_k \in V^*$ by
\[
\langle \lambda_k, v\rangle_{V^*,V} := 2\beta \int_\Omega \psi_{\epsilon_k}'(u_k^2)u_{k+1}v\dx  \quad v\in V.
\]
Due to the boundedness properties of \cref{lem_iter_bounded} and \cref{coro_Lk_bounded}, it follows that $(\lambda_k)$ is bounded in $V^*$.

Let $u_{k_n} \rightharpoonup \bar u$ in $V$.
After extraction of a subsequence if necessary, we can assume $\lambda_{k_n} \rightharpoonup \bar \lambda $ in $V^*$.
Then we can pass to the limit along the subsequence in \eqref{eq_optcon_iter} to obtain \eqref{eq_optcon_lim_iter}.
Here, we used complete continuity of $F'$ and \cref{coro_summing}.

Due to \cref{lem_iter_lambdau}, we have $\lambda_{k_n}(u_{k_n+1})  \to  p \int_\Omega |\bar u|^p \dx$. Testing \eqref{eq_optcon_iter} with $u_{k_n+1}$ yields
\[
 \alpha \| u_{k_n+1}\|_V^2 + L_{k_n} \langle u_{k_n+1}-u_{k_n},u_{k_n+1}\rangle_V + F'(u_{k_n})u_{k_n+1} = - \beta \lambda_{k_n}(u_{k_n+1}).
\]
Passing to the limit inferior results in
\[
 \alpha \| \bar u\|_V^2 + F'(\bar u)\bar u  \le   -p \int_\Omega |\bar u|^p \dx .
\]
The left-hand side is equal to $-\beta \langle \bar\lambda, \bar u\rangle_{V^*,V} $ by \eqref{eq_optcon_iter}, which proves the claim.
\end{proof}

The system satisfied by limits of the iteration in \cref{thm_optcon_lim_iter} is clearly weaker than the system provided by \cref{thm_opt_cond}.
This is due to the fact that we cannot expect strong convergence of the iterates of the method,
which would be necessary to pass to the limit $\langle \lambda_{k_n},u_{k_n+1}\rangle_{V^*,V} \to \langle \bar\lambda, \bar u\rangle_{V^*,V} $.

% \printbibliography

\bibliography{hslp.bib}
\bibliographystyle{plain}

\end{document}